\documentclass[a4paper,oneside,12pt]{article}

\usepackage{amsmath,amsfonts,amscd,amssymb,amsthm,graphicx,bbm,stmaryrd,color,xspace}
\usepackage[bf,small,margin=1cm]{caption}
\usepackage[margin=1in]{geometry}
\usepackage[utf8]{inputenc}
\usepackage{microtype}
\definecolor{grayblue}{rgb}{0.09,0.32,0.44} 
\usepackage[colorlinks,linkcolor=grayblue,citecolor=grayblue,urlcolor=blue,pdfborder={0 0 0}]{hyperref}
\usepackage{breakurl}

\newtheorem{theorem}{Theorem}[section]
\newtheorem{proposition}{Proposition}[section]
\newtheorem{lemma}{Lemma}[section]

\newtheorem{corollary}{Corollary}[section]

\newtheorem*{definition*}{Definition}

\theoremstyle{remark}

\newtheorem*{remark*}{Remark}

\newcommand{\be}[1]{\begin{equation}\label{#1}}
\newcommand{\ee}{\end{equation}}
\newcommand{\ban}{\begin{align*}}
\newcommand{\ean}{\end{align*}}
\newcommand{\ba}[1]{\begin{align}\label{#1}}
\newcommand{\ea}{\end{align}}
\newcommand{\ben}{\begin{equation*}}
\newcommand{\een}{\end{equation*}}
\numberwithin{equation}{section}

\newcommand{\bbE}{\mathbb{E}}

\newcommand{\bbP}{\mathbb{P}}

\newcommand{\bbR}{\mathbb{R}}

\newcommand{\bbZ}{\mathbb{Z}}
\newcommand{\sfA}{{\sf A}}

\newcommand{\ep}{\varepsilon}

\newcommand{\IDLA}{\ensuremath{\mathrm{IDLA}}\xspace}

\newcommand{\rIDLA}{\ensuremath{\mathrm{uIDLA}}\xspace}
\newcommand{\DA}{D\hspace{-1pt}A}
\newcommand{\A}{A}
\newcommand{\T}{\ensuremath{\mathcal{T}}}
\newcommand{\prob}{\ensuremath{\mathbb{P}}}
\newcommand{\Add}{{\rm Add}}

\DeclareMathOperator{\Bin}{Bin}

\newcommand{\rk}[1]{\bgroup\color{red}%
  \par\medskip\hrule\smallskip%
  \noindent\textbf{#1}%
  \par\smallskip\hrule\medskip\egroup}

\ifpdf
\def\afs#1#2{\href{#1}{\nolinkurl{#2}}}
\else
\def\afs#1#2{\burlalt{#1}{#2}}
\fi

\newcounter{const}
\newcounter{Const}
\def\newCnotext#1{
\refstepcounter{Const}
\label{#1}
}
\def\cref#1{\ensuremath{\textrm{\textup{\ref{#1}}}}}
\def\newC#1{{\newCnotext{#1}\cref{#1}}}
\def\newcnotext#1{
\refstepcounter{const}
\label{#1}
}
\def\newc#1{{\newcnotext{#1}\cref{#1}}}

\title{Internal diffusion-limited aggregation with uniform starting points}
\author{Itai Benjamini \and Hugo Duminil-Copin \and Gady Kozma \and Cyrille Lucas}
\date{\today}

\begin{document}
\maketitle

\begin{abstract}
We study internal diffusion-limited aggregation with uniform starting points on $\mathbb Z^d$. In this model, each new particle starts from a vertex chosen uniformly at random on the existing aggregate. We prove that the limiting shape of the aggregate is a Euclidean ball.

\bigskip

Nous étudions le modèle d'agrégation limitée par diffusion interne avec points de départ uniformes sur $\mathbb Z^d$. Dans ce modèle, chaque nouvelle particule est ajoutée à un point choisi uniformément au hasard parmi ceux de l'agrégat existant. Nous prouvons que l'agrégat normalisé admet comme forme limite la boule euclidienne.

\emph{Mots-Clés :} Modèle de croissance, Marche aléatoire, IDLA, mesure harmonique.
\end{abstract}

\section{Introduction}

\subsection{Historical introduction and motivation}

Internal diffusion-limited aggregation (\IDLA)  was introduced by Diaconis and Fulton in \cite{DF}, and gives a protocol for recursively building a random aggregate 
of particles. At each step, the first vertex visited outside the current aggregate by a random walk started at the origin is added to the aggregate. In a number of settings, this model is known to have a deterministic limit-shape,
meaning that a random aggregate with a large number of particles has a typical shape.
On $\bbZ^d$, Lawler, Bramson and Griffeath \cite{lawler1992internal} were the first to identify this limit-shape, in the case of simple random walks, as the Euclidean ball. Their result was later sharpened by Lawler \cite{lawler1995subdiffusive}, and was recently drastically improved with the simultaneous works of Asselah and Gaudill\`ere \cite{asselah2010logarithmic, asselah2010sub} and Jerison, Levine and Sheffield \cite{jerison2010internal,jerison2010logarithmic,jerison2011internal}, where logarithmic bounds are proved for fluctuations of the boundary. 

The \IDLA model has been extended in several contexts including drifted random walks \cite{Luc}, Cayley graphs of finitely generated groups \cite{Bla04,BB07,DLY13,Hus08} and random environments \cite{DLYY11,She10}. 

Another interesting growth model is provided by Richardson's model \cite{Ric73}, which is defined as follows. At time $0$, only the origin is occupied. A vacant site becomes occupied at an exponential time with a rate proportional to the number of occupied neighbours, and once occupied a site remains occupied.  The set of vertices occupied by time $t$ is the ball of radius $t$ centered at the origin in first passage percolation with exponential clocks (see \cite{Kes86}). Eden \cite{Ede61} first asked about the shape of this process on Euclidean lattices and Richardson proved that a limiting shape exists. It is believed that the convex centrally symmetric limiting shape is not a Euclidean ball. This was established by Kesten in high dimensions (unpublished, but see \cite{CEG11}) together with the fact that the boundary has $t^{1/3}$ fluctuations, a long standing conjecture. 

Internal diffusion-limited aggregation with uniform starting points (from here on shortened to \rIDLA) is a growth model interpolating between standard internal diffusion-limited aggregation and Richardson's model. In \rIDLA, particles are born uniformly on the shape and relocate to the outer boundary according to harmonic measure seen from the site they appeared at. While usual \IDLA approaches rely on estimating the number of visits to a given point by particles starting from the origin, either directly or as the solution to a discrete partial differential equation, the study of \rIDLA is more difficult because of the self-dependence involved in the construction.   

Another related model is \emph{excited to the center}. In this model a single particle walks around the lattice $\bbZ^d$ doing simple random walk, except when it arrives at a vertex it never visited before (``a new vertex''), in which case it gets a drift towards the point $0$. To compare excited to the center to the models described so far, think about standard \IDLA as a single particle which, upon reaching a new vertex, is teleported to $0$; and about \rIDLA as a single particle which, upon reaching a new vertex, is teleported to a random location in the visited area. Very little is known about random walk excited to the center --- there is an unpublished result showing that it is recurrent in all dimensions, but the shape of visited vertices is very far from being understood. Simulations and some heuristics indicate that at time $t$ the set of visited vertices should be a ball with radius approximately $t^{1/(d+1)}$.

For \rIDLA, we show that the limiting shape is a Euclidean ball, hence showing a behaviour close to the standard \IDLA behaviour. Yet, the boundary fluctuations are expected to be slightly stronger than that of standard \IDLA. This is not surprising, since part of the growth is due to particles emerging near the boundary thus behaving very roughly like the Richardson model. This suggests that the local regularity will be determined by some competition between particles born locally \textit{\`a la} Richardson and particles arriving from far away as in standard \IDLA. Furthermore, simulations like the one we present below seem to indicate a mesoscopic shift in the center of mass of the cluster, which occurs in a random direction. This paper deals with the limiting shape and not the fluctuations.

\begin{figure}
\begin{center}
\includegraphics[scale=0.29]{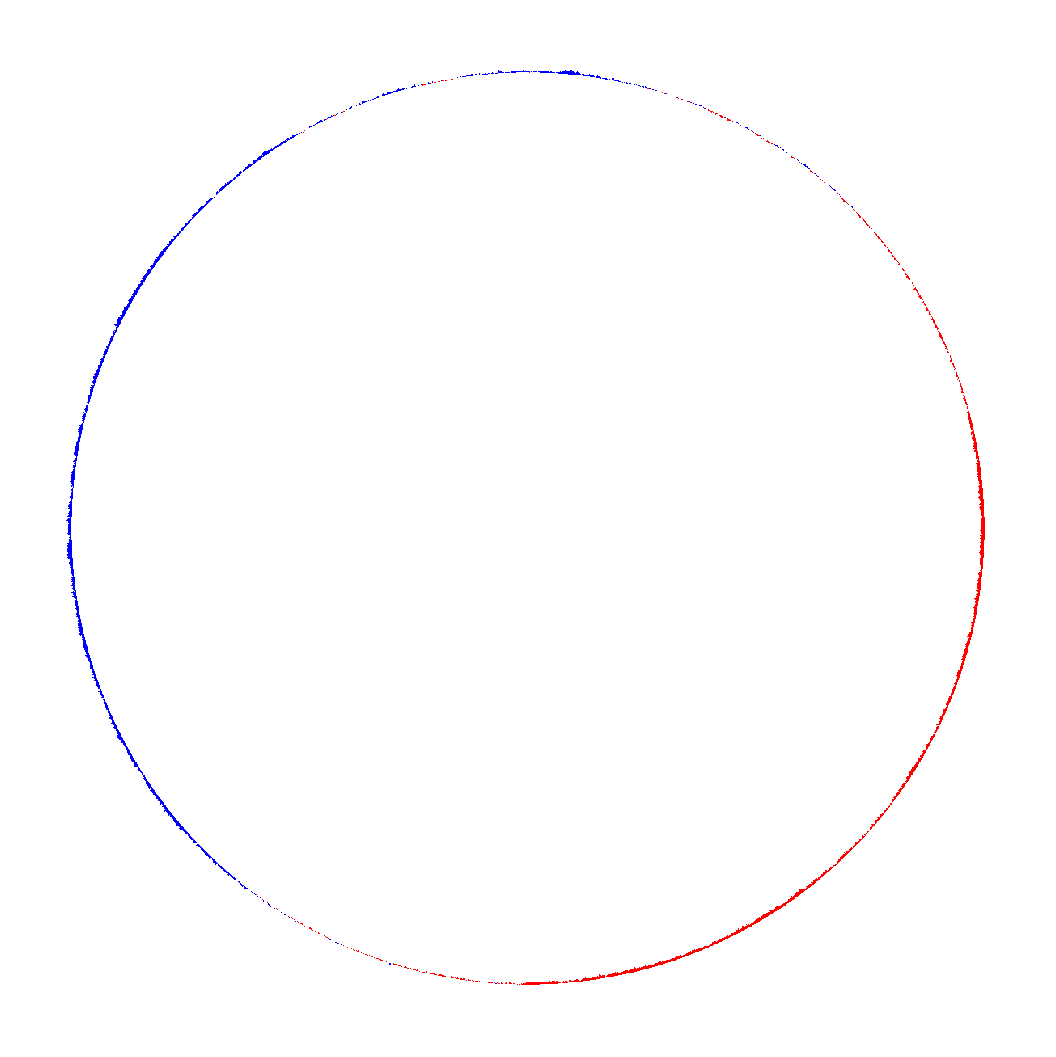}
\caption{Symmetric difference between the two-dimensional \rIDLA aggregate and the Euclidean ball, with $10^6$ particles. Blue points are present in the aggregate but not the ball, whereas it is the other way around for red points.}
\end{center}

\end{figure}

\subsection{Definition of the model and statement of the main theorem}
We consider the lattice $\bbZ^d$ with $d\ge 1$. Let $S \subseteq \bbZ^d$ be a finite subset of $\bbZ^d$. 

In order to define both standard and uniform starting point \IDLA, first define the action of adding a particle to an existing aggregate $S$. Let $\xi = (\xi(0),\xi(1),\ldots)$
be a random walk on $\bbZ^d$ and let $t_S$ be the first time this walk is not in $S$. By random walk we mean the simple random walk choosing one of its $2d$ neighbours uniformly and independently at random at each step.  Define
$$\Add[\xi,S]:=S\cup\{\xi(t_S)\}.$$

\paragraph{Standard IDLA}

Fix an integer $n \geq 0$.
Let $\DA_n$ be the aggregate with $n$ particles started at $0$, constructed inductively as follows: $\DA_0=\emptyset$ and
$$\DA_{n+1}:=\Add[\xi_{n}^0,\DA_n]$$ where $\xi^0_{n}$ is a random walk starting at $0$ which is independent from $\xi^0_0,\ldots,\xi^0_{n-1}$. 
This process is referred to as \IDLA.

Note that the equivalent initialisation $\DA_1 = \{0\}$ is sometimes used.

\paragraph{IDLA with uniform starting point} Fix an integer $n\geq 0$. Let $\A_n$ be the uniform starting point aggregate with $n$ particles constructed inductively as follows: $\A_1=\{0\}$ and
$$\A_{n+1}:=\Add[\xi^{X_{n}}_{n},\A_n]$$
where $X_n$ is a point chosen uniformly on $\A_n$, and $\xi^{X_{n}}_{n}$ is a random walk starting at $X_n$ and independent of $\xi^{X_0}_0,\ldots,\xi^{X_{n-1}}_{n-1}$. This process is referred to as \rIDLA.
\medbreak
\noindent Let $|\cdot|$ be the Euclidean distance in $\bbR^d$. For $n>0$, let $B[n]:=\{y\in \bbZ^d:|y|\le n\}$ and $b_n:=|B[n]|$. 

\begin{theorem}\label{thm:main}
Let $d\ge 2$. There exists positive constants $c_\newc{c:thm}$, $c_\newc{c:thm2}$, $C_\newC{C:Thm}$ and $C_\newC{C:Thm2}$ depending only on the dimension, such that almost surely, 
$$B\Big[n(1-C_\cref{C:Thm}n^{-c_\cref{c:thm}}) \Big] \subseteq \A_{b_n}\subseteq B\Big[n(1+C_\cref{C:Thm2}n^{-c_\cref{c:thm2}})\Big]$$
for $n$ large enough.
\end{theorem}

\paragraph{Remarks} In dimension $1$, the \rIDLA\ aggregate with $n$ points is a set of consecutive integers of length $n$, therefore it is entirely determined by the position of its middle point (called $M_n$). It is clear, either from a quick computation using the gambler's ruin or from a symmetry argument, that the probability for the cluster to grow on either of the two sides is exactly $1/2$. Therefore the process $M_n$ is exactly a simple random walk on integers and half-integers, and the behaviour of the cluster is obvious, with a law of large numbers and CLT fluctuations.

In dimensions bigger than $2$ we expect much smaller fluctuations, and our theorem is not satisfactory in this regard.  We have chosen not to optimize the $n^{-c}$, mainly in order to help alleviate notations, but also because we do not hope to capture the true order of the error term with our method.

\subsection{Structure of the paper}
The first section contains five lemmas. They provide useful information on comparing \rIDLA to \IDLA. As they are of interest on their own, we isolate them from the proof of the theorem.

The second section of the article deals with the stability properties of the Euclidean ball under the \rIDLA process. We first investigate the claim that the process started from a configuration that includes a ball will contain a growing ball with high probability. Then we take the converse and prove that the process started from any configuration inside a ball will stay contained in a slightly bigger growing ball. To prove this statement we examine the cluster together with the genealogical tree describing the starting points of the random walks. 
Our proof involves a comparison with a First Passage Percolation process on random trees.

In our third section, we bring these elements together for a proof of our theorem. The inner bound is proved first, using a refinement method that relies heavily on our coupling properties. The outer bound is then proved using the genealogical construction from the previous section.

\paragraph{Further notation}
For every $y\in \bbZ^d$, let $\prob_y$ denote the law of a simple random walk on $\bbZ^d$ starting from $y$. For a set $S\subset\bbZ^d$ we will denote by $\partial S$ the set of vertices in $\bbZ^d\setminus S$ with a neighbour (or more than one) in $S$.



\section{Comparison lemmas}

We start with the following notations which will enable us to state our lemmas more easily. Given vertices $x_1,\ldots,x_k$ in $\bbZ^d$, define $\DA_{x_1,\ldots,x_k}(S)$ to be the \IDLA aggregate formed by launching additional particles from points $x_1,\ldots,x_k$. Note that $x_i$ need not be in the set $S$. Naturally, for $x\not\in S$, $\DA_{x}(S)=S\cup\{x\}$ deterministically. Recall that, classically, the law of the aggregate does not depend on the order in which these particles are added. Therefore, we also define $\DA_X(S):=\DA_{x_1,\ldots,x_k}(S)$, where $X$ is the multi-set $X=\{x_1,\ldots,x_k\}$. If the multi-set $X$ is just $k$ repetitions of the origin, we denote for conciseness $\DA_k(S):=\DA_X(S)$. Remark that this notation is consistent with our initial definition, in that $\DA_n=\DA_n(\emptyset)$.

Similarly, for the \rIDLA process, we denote $\A_1(S)$ the result of adding a particle started uniformly on $S$ to the set $S$. We also denote $\A_k(S)$ for $k\in \mathbb{N}$ the result of the recursive process of adding $k$ particles to the aggregate $S$, where the first one starts uniformly on $S$, and the $j$-th particle starts uniformly on $\A_{j-1}(S)$.

We start with the following lemma. It states that the aggregate obtained by launching $k$ particles from arbitrary points in $ B[n/2]$ is bigger than the aggregate obtained by launching a smaller yet comparable number of particles from the origin.

\begin{lemma}\label{win some}
There exists $\eta>0$ (depending only on the dimension) such that for any multi-set $X$ of cardinality $k$ in $ B[n/2]$, $\DA_X( B[n])$ stochastically dominates $\DA_{ \kappa }( B[n])$, where $\kappa$ follows a binomial distribution $\mathcal{B}(k, \eta)$.
\end{lemma}

\begin{proof}
Let $x$ be a point in $B[n/2]$ and $A$ a set containing $B[n]$. We consider the function evaluating the probability that the random walk starting at $x$ exits $A$ through a point $y$. 
We consider the stopping time $\tau_{A}=\inf \left\{ t: \xi^x(t) \notin A \right\}$ and the function

$$h_y(x) = \prob_x \left(\xi^{x}(\tau_{A}) = y \right). $$

This function is harmonic in $x$ on $B[n-1]$, hence the Harnack inequality \cite[Theorem 6.3.9]{LL10} implies that there exists $\eta>0$ such that for all $A \supseteq B[n], x \in B[n/2]$ and $y \notin A$,

$$h_y(x) \geq \eta h_y(0).$$

This inequality allows to construct a coupling between $\DA_X$ and $\DA_\kappa$ as follows. Let $E_k$ and $F_k$ be constructed recursively. Set $E_0= F_0=B[n]$. Index sites of $X$ by $\{x_1,\dots,x_{|X|}\}$. Assume that $F_k\subseteq E_k$ have been constructed. Construct $E_{k+1}=\Add[\xi_{k+1}^{x_{k+1}},E_k]$. Consider a killed random walk $\xi^0_{k+1}$ coupled with $\xi_{k+1}^{x_{k+1}}$ in such a way that:
\begin{itemize}
\item $\xi_{k+1}^0$ is killed at 0 with probability $1-\eta$.
\item if $\xi_{k+1}^0$ exit $E_{k}$ through $y$, so does $\xi_{k+1}^{x_{k+1}}$, 
\end{itemize}
The existence of this coupling is guaranteed by $h_y(x) \geq \eta h_y(0)$. After exiting $B[n]$ we couple the walks in the usual way: they walk together until exiting their respective aggregates (since $F_k\subseteq E_k$, the walk on $F_k$ would exit first). Construct $F_{k+1}=\Add[\xi_{k+1}^0,F_k]$ if the particle is not killed. Note that $F_{k+1}\subseteq E_{k+1}$, since either $\xi_{k+1}^0$ exits $F_k$ throughout a point of $E_k$, or it does through a point not in $E_k$, but in this case the coupling guarantees that the exiting point is in $E_{k+1}$. The total number of coupled particles, $\kappa$, follows a binomial distribution with parameters $(k, \eta)$.
\end{proof}

The following lemma controls the behaviour of a standard \IDLA with $M$ points started with a Euclidean ball $B[n]$ already occupied. It closely follows the spirit of \cite{asselah2010logarithmic}, but instead of pushing the precision to get the best almost sure bound, we only look at points that are far enough from the edge of the theoretical shape to keep the probability of inclusion exponentially close to $1$. 
\begin{lemma}\label{st-exp-bound}
For any $n \in \mathbb{R}^+$ and $N \in \mathbb{N}^*$, let us write $r=r_{n,N}=\frac{N}{b_n}.$ Then we have $$ \mathbb{P}\left( B\left[n \left(1+r-r^{3/2}\right)^{1/d}\right] \subset \DA_{ N }\big(B[n] \big) \right) \geq 1 -  \exp\left(-C_{\newC{C:st-xp-bound}}nr^{3/2}\right).$$
\end{lemma}
We will use this lemma in the window where $n r^{3/2}$ is large but $r$ is small, so that $r-r^{3/2} \geq 0$ (otherwise the lemma is true but useless).

\begin{proof}
Our lemma is almost exclusively a consequence of the many ideas provided in \cite{asselah2010logarithmic}. Therefore, we refer the reader to the appendix, in which we give a guide to the modifications one needs to do in \cite{asselah2010logarithmic} to get this result.

\end{proof}

The next two lemmas propose stochastic dominations between standard \IDLA and the \rIDLA process. We start with a lemma that compares one step. 

\begin{lemma}\label{stochastic domination}

There exists a constant $C_\newC{C:A dom DA}>0$ such that, if $B[n]\subset S\subset T$ then $\A_1(T)$ stochastically dominates $\DA_{\delta}(S)$, where $\delta$ is a Bernoulli variable with parameter $\frac{|B[n]|}{|T|}(1-\frac{C_\cref{C:A dom DA}}{n})$.

\end{lemma}

\begin{proof}
First, remark that our new point falls inside $B[n]$ with probability $\frac{|B[n]|}{|T|}$.
Once more, we consider the stopping time $\tau_{A}=\inf \left\{ t: \xi^x(t) \notin A \right\}$ and the function
$$h_y(x) = \prob_x \left(\xi^{x}(\tau_{A}) = y \right). $$
This function is harmonic in $x$ on $B[n]$. We are now interested in an averaging property for this harmonic function; namely, is $h_y(0)$ close to $\frac{1}{|B[n]|} \sum_{x \in B[n]} h_y(x)$ ?

The study on this averaging property is linked to that of \emph{quadrature domains} and the divisible sandpile model, and, in particular, one shape on which a relation is known between the two terms is the shape taken by the divisible sandpile after toplings, with all the initial mass started at the origin, as defined in \cite{levine2009strong}. Let $m(x)$ be the final mass distribution corresponding to an initial mass $M$ at the origin, then we have, for all harmonic functions $h$,
$$ M h(0) = \sum_{x \in \bbZ^d} m(x) h(x).$$
Recall that the final mass distribution $m$ is equal to $1$ on a given shape, has value between $0$ and $1$ at distance one from this shape, and is zero at distance more than one of this shape.  It is hence a consequence of Levine and Peres's shape theorem (see \cite{levine2009strong}) that there is a constant $c>0$ depending only on the dimension such that
$$ h_y(0) = \frac{1}{|B[n]|} \sum_{x \in \bbZ^d} m(x) h_y(x), $$
with $m(x)=1$ on $B[n-c]$ and $m(x)=0$ outside $B[n+c]$. Combining the facts that $m$ has values between $0$ and $1$ everywhere; and that $\sum_{y} h_y(x)=1 $, allows to bound the error given by replacing $m$ with $\mathbbm{1}_B$. We get that there is a constant $C_\cref{C:A dom DA}$ depending only on the dimension, such that:
$$ \sum_{y \in \partial A} \Big| h_y(0) - \frac{1}{|B[n]|} \sum_{x \in B[n]} h_y(x) \Big| \leq \frac{C_\cref{C:A dom DA}}{n}. $$

Hence, our two particles can be coupled with probability $\frac{|B[n]|}{|A|}(1-\frac{C_2}{n})$, which yields the result.\end{proof}


Assume $E$ is some subset of our aggregate $F$. As $F$ evolves, there is a natural increasing subset $E_n\subset \A_n(F)$ which corresponds to $E$ and is in fact a time change of an \rIDLA started from $E$. Basically, one traces only particles which started on $E_n$ and follows them only until they exit $E_n$. Further, it is not necessary to know anything about the structure of $F$, it is enough to know its size. Formally, the definition is as follows: Let $E_0=E$. Next, for every $n$ define
$$E_{n+1}:=\begin{cases}\Add[\xi^{X_n}_{n+1},E_n]& \text{with (independent) probability $\frac{|E_n|}{|F|+n}$ } \\
E_n&\text{otherwise}\end{cases}$$
where $X_n$ is a point chosen uniformly on $E_n$, and $\xi^{X_n}_{n+1}$ is a random walk starting at $X_n$ and independent of $\xi^{X_1}_1,\ldots,\xi^{X_{k-1}}_k$. Finally, the Bernoulli events which determine whether the point will be added or not are independent of the walks (and of one another). We see that the process depends only on the size of $F$ and not on its structure. This leads to the following definition
\begin{definition*}
  For $E\subset \bbZ^d$ and $m\ge |E|$ we let $\A_n(E;m)$ be the $E_n$ defined in the previous paragraph, for some $F$ with $|F|=m$. We call $\A_n(E;m)$ the subset \rIDLA.\label{page:censored}
\end{definition*}

Clearly $\A_n(E;|E|)$ is the same as $\A_n(E)$ and, in general, if $E\subset F$, then $\A_n(F)$ stochastically dominates $A_n (E;|F|)$. A little more than that is, in fact, true:

\begin{lemma}\label{crucial lemma}
For any sets $E\subseteq F$, we have that $\A_n(F)$ stochastically dominates $\DA_{ F \setminus E }\big(\A_n(E;|F|)\big)$.
\end{lemma}

\begin{proof}
We will colour $\A_n(F)$ in 3 colours, blue, red and black, such that the blue part has the same distribution as $\A_n(E;|F|)$, the union of the red and the blue has the same disribution as $\DA_{F\setminus E}(\A_n(E;|F|))$ and black is the rest. Here is the colouring scheme:

We start the process with $A_0=F$ coloured as follows: $E$ is coloured blue and $F\setminus E$ is coloured red. Suppose we already constructed (and coloured) $A_n$. We choose a vertex $x$ of $A_n$ randomly to start the random walk from. 
\begin{itemize}
\item If $x$ is blue, perform the random walk until the particle exits the blue set. When it does, the site where it lands is coloured blue. If there was already a particle at that site, ``wake it up'' --- it continues walking according to the rules in the following clauses.
\item Now assume we have a red particle walking (which can only happen if a red particle was woken by a blue one, as in the previous clause). Perform the random walk until the particle exits the union of the red and the blue. When it does, that site will be coloured red. If there is a black particle there, wake it up and let it continue walking according to the rule in the next (and last) clause.
\item If $x$ is red or black, let the new particle be black. Let it perform simple random walk until the it exits the entire aggregate, and colour that site black.
\end{itemize}
Thus, for example, a particle might start from a blue site, walk until reaching a red site, change that site to blue, continue walking until reaching a black site, change that site to red, and then walk until exiting. This ends the description of the colouring.

Now, the fact that the blue part of the aggregate has the same distribution as $\A_n(E;|F|)$ is evident. The fact that the union of the red and the blue has the same distribution as $\DA_{F\setminus E}(\A_n(E;|F|))$ is also simple, because the red part starts with $F\setminus E$ and then each red particle does a random walk and ends outside the eventual blue part.

One might claim that, even though each red particle does simple random walk, they are stopped and woken up mixing up their order. It is well-known that this does not affect the distribution of the final aggregate. For the convenience of the reader, let us recall the argument. One attaches labels to each red particle, and when a particle with a lower label steps over a particle with a higher label, they exchange labels so that the higher label continue to walk. This, of course, does not change the red part, but now each label does simple random walk until its final resting point, and only then does the next label start to walk. So the union of the blue and the red part has indeed the same distribution as $\DA_{F\setminus E}(\A_n(E;|F|))$ and the lemma is proved.
\end{proof}

The following lemma is extracted from \cite{DLYY11}. It states that a random walk has a small probability of passing through an area of small density, and will be used to couple our process with a First Passage Percolation process. Rather than refer to the proof of \cite{DLYY11} which holds in a more general setting, we give a shorter proof specific to $\mathbb{Z}^d$. Recall that we defined $b_n=|B[n]|$ the volume of the Euclidean ball of radius $n$ intersected with $\mathbb{Z}^d$.

\begin{lemma}\label{lem:exit prob}
Let $p>0$. There exists $\ep>0$ such that for any $n,m\ge 1$ large enough,  
$$\prob_x\Big(\xi\text{ \rm exits }S\cup  B[m]\text{ \rm through }\partial  B[m+n]\Big) \le p$$
uniformly in $x\in  B[m]$ and $S \subseteq  B[m+n]$ satisfying $|S|\le \ep b_n$.\end{lemma}

\begin{proof}
By Markov's property, it is sufficient to bound $$\prob_y\left(\xi\text{ \rm exits } S \cap B_y\left[n/3\right] \text{ \rm through } \partial B_y\left[n/3\right] \right) $$ for starting points $y \in S\cap B[m+\frac {2n}3]\setminus B[m+\frac n3]$. Similarly, by shifting $y$ to zero and replacing $n/3$ by $n$, it is enough to prove that
$$\prob_0\left(\xi\text{ \rm exits } S \cap B[n] \text{ \rm through } \partial B[n]\right) \leq p$$
uniformly in any set $S\subset  B[n]$ such that $|S|\le \ep b_n$, for $\ep$ small enough (thus our new $\ep$ is multiplied by $3^d$, which does not affect the rest of the proof). 

Now, if $|S|\le\ep b_n$ then for some $r\le n$ we must have that $S\cap\partial B[r]\le C\ep r^{d-1}$. 
By \cite[Lemma 6.3.7]{LL10}, every $x\in\partial B[r]$ has probability $\le Cr^{1-d}$ that random walk started from $0$ will exit $B[r]$ at $x$. Summing over $x\in S\cap \partial B[r]$ gets that the probability that random walk started from $0$ will exit $B[r]$ at $S$ is less than $C\ep$. Thus for $\epsilon$ small enough, $\xi$ exits $S$ before reaching $\partial B[n]$ with probability greater than $1-p$.
\end{proof}


\section{Stability of the Euclidean ball}\label{sec:stability}

\subsection{Inner stability of the ball}\label{subsec:inner}


In this section (\S \ref{sec:stability}) we show that, if you start a \rIDLA from a large ball, it remains an approximate ball, with high probability. We first (\S\ref{subsec:inner}) show inner stability, i.e.\ that the aggregate \emph{contains} a ball of the approximately correct size. In a formula,
\[
\A_{b_m-b_n}(B[n])\supseteq B[m(1-Cn^{-1/4})] 
\]
with high probability. In other words, the only error is the missing $Cn^{-1/4}$ in the diameter. 

It will be convenient, though, to formulate the claim slightly more generally: if $B[n]\subseteq S$ then $\A_M(S)$ contains a ball of the correct size. We will use the notation $\A_M(E;N)$ introduced on page  \pageref{page:censored} --- recall that $\A_M(E;N)$ is the way $E$ evolves when you embed it in some set of size $N$, add $M$ particles in a \rIDLA fashion to that set, and examine only particles that landed on $E$. We first formulate a lemma for adding a relatively small number of particles, an $n^{-1/2}$ proportion:


\begin{lemma}\label{not loosing too much}
There exist $\delta_2,C_\newC{C:innerstability}>0$ such that for any $M\ge b_n$,
$$
\bbP\Big(B\big[n\big(1+n^{-1/2}-C_\cref{C:innerstability}n^{-3/4}\big)^{1/d}\big]
                \subseteq \A_{Mn^{-1/2}}(B[n];M)\Big)
\ge 1-C\exp(-n^{\delta_2}).$$
Remark in particular that the probability does not depend on $M$.
\end{lemma}

\begin{proof}
The definition of $\A_{Mn^{-1/2}}(B[n];M)$ gives that it is the same as $\A_K(B[n])$ where $K$ is a random variable which stochastically dominates a binomial distribution with $Mn^{-1/2}$ trials and probability $b_n/(M+Mn^{-1/2})$ for success. 

Recall that Lemma~\ref{stochastic domination} says that adding a single particle to \rIDLA stochastically dominates adding a single particle to standard \IDLA, with an appropriate probability. Applying Lemma~\ref{stochastic domination} $K$ times 
gives that $\A_{K}(B[n])$ stochastically dominates a standard \IDLA with initial set $B[n]$ and with a random number $L$ of particles (started at the origin), where $L$ stochastically dominates a binomial distribution with $K$ trials and probability $(b_n/(b_n+K))(1-C_ \cref{C:A dom DA}/n)$ for success.


As a first step we need to make sure that $K$ is not too large, so that the factor $(b_n/(b_n+K))$ does not impact the probability for success too much. Since the expected value of $K$ is $n^{-1/2}b_n/(1+n^{-1/2})$, and $b_n=\omega_d n^d + O(n^{d-1})$, we know that $K$ should only be of order $n^{d-1/2}$. We use a Chernoff bound to ensure that:
$$ K \leq \frac{2n^{-1/2}b_n}{(1+n^{-1/2})}$$ with probability larger than $1-C\exp(-n^{d-1/2}/3)$ (note that $d\geq 2$). We now assume that this bound for $K$ is verified, and therefore $b_n/(b_n+K) \geq 1/(1+2n^{-1/2})$.

Combining these facts shows that $\A_{Mn^{-1/2}}(B[n];M)$ stochastically dominates standard \IDLA started from a ball with the number of particles $L$ following a binomial distribution $\Bin(s,p)$  with
$$s=Mn^{-1/2}\quad\text{and}\quad p= \frac{b_n}{M(1+n^{-1/2})(1+2n^{-1/2})}\left(1-\frac{C_\cref{C:A dom DA}}{n}\right).$$
Since $b_n=\omega_d n^d + O(n^{d-1})$, another Chernoff bound directly yields that $L$ satisfies:
$$ \left| L - \frac{n^{-1/2}b_n}{(1+n^{-1/2})(1+2n^{-1/2})}\left(1-\frac{C_\cref{C:A dom DA}}{n}\right) \right|\leq n^{d-3/4}$$ with probability larger than $1-\exp(-Cn^{d-1})$, for a constant $C>0$ (note that $d\geq 2$). We now assume that this bound for $L$ is verified.

We then apply Lemma \ref{st-exp-bound} with $B[n]$ already occupied and $L$ new particles started at the origin. With the notations of the lemma, $r_{n,N} \sim n^{-1/2}$ and we get that the ball of radius $n(1+r-r^{3/2})^{1/d}$ is included in the cluster $\DA_{L}(B[n])$ with probability at least $1-\exp(-C_{\cref{C:st-xp-bound}}n^{1/4})$. We estimate \newcnotext{c:3} 
\begin{align*}1+r-r^{3/2}
&\ge 1 + \frac{n^{-1/2}}{(1+n^{-1/2})(1+2n^{-1/2})}\left(1-\frac{C_\cref{C:A dom DA}}{n}\right) -\frac{n^{d-3/4}}{b_n} -Cn^{-3/4} \\
&\ge 1 + n^{-1/2}-C_\cref{C:innerstability}n^{-3/4}
\end{align*}
where the first inequality is the lower bound on $L$, and these inequalities hold with probability $1-C\exp(-n^{\delta_2}).$  \newcnotext{c:4}\newcnotext{c:5}\newcnotext{c:6}
\end{proof}

The case where the number of particles we add is proportional to the volume (or more) is a corollary:

\begin{corollary}\label{lose some}
There exist $\delta_3,C_\newC{C:inner2}>0$ such that for any $M\ge b_n$, 
\begin{multline*}
\bbP\Big(\forall a\ge 1\; B\big[ n  a^{1/d} \big(1-C_\cref{C:inner2}n^{-1/4} \big) \big]\subseteq \A_{(a-1)M}(B[n];M)\Big)\\
\ge 1-C\exp(-n^{\delta_3}).
\end{multline*}
\end{corollary}

\begin{proof}
Examine first the case that $a\leq 2$.
  We apply the previous lemma repeatedly $K$ times, i.e.\ define 
\begin{center}
\begin{tabular}{ll}
$S_0=B[n],$ & $M_0=M,$ \\
$S_{i+1}=\A_{M_in^{-1/2}}(S_i;M_i),$ & $M_{i+1}=M_i(1+n^{-1/2}),$
\end{tabular}
\end{center}
with $K$ chosen in such a way that we obtain additional $|S|$ particles. Since each time we add $M_in^{-1/2}$ particles and $M_i\ge M$, we deduce that $K\le n^{1/2}$ 
Therefore, with probability larger than 
$$1-n^{1/2}\exp(-n^{\delta_2})$$
the aggregate $\A_{(a-1)M}(B[n];M)$ contains the Euclidean ball of radius
\begin{align}n a^{1/d}\left(1-Cn^{-3/4}\right)^{K}&\ge n a^{1/d}\left(1-Cn^{-3/4}\right)^{n^{1/2}}\nonumber\\
&\ge n a^{1/d} \Big( 1-C n^{-1/4} \Big).\label{eq:radius}
\end{align}
This takes care of $a$ along a sequence. For a general $a\in[1,2]$, we find some $i$ such that $M_i<(a-1)M<M_{i+1}$ and the inequality still holds from monotonicity of the aggregate (we lose $Cn^{d-1/2}$ particles from the approximation, but this only changes the value of the constant in (\ref{eq:radius})).

For general $a$ (i.e.\ $a>2$) we repeat the last calculation for $2n$, $4n$ etc. We get that the claim holds for all $a$ except for an event whose probability is smaller than 
$$
C\sum_{i=0}^n \big(2^in\big)^{1/2}\exp(-2^in^{\delta_2}).
$$
Since this sum converges exponentially, we may bound it by $C\exp(n^{-\delta_3})$ for an appropriate $\delta_3$. Similarly, the errors in (\ref{eq:radius}) converge exponentially, so they only change the constant. So we get that the radius is bounded by
$$
na^{1/d}(1-C_\cref{C:inner2}n^{-1/4})
$$
for a suitable constant $C_\cref{C:inner2}$. 
%
\end{proof}

\subsection{Genealogical construction and outer stability}
\label{genealogy}
Our aim in this section is to prove a converse to Corollary \ref{lose some} for the outer stability of the ball. We begin by comparing the process started from any set $S \subset B[n]$ with the process started from $B[n]$, on an event of high probability. Here we are comparing \rIDLA to another \rIDLA (and not to standard \IDLA, as in the previous section), so the argument is much simpler.

\begin{lemma} \label{lem:increasing}
There exist $\delta_5>0$ and $C_\newC{C:6} >0$ such that for any set $S$ with $S \subseteq B[n]$, and for any $1\leq a \leq 2$, there is a coupling of $\A_{(a-1)|S|}(S)$ and $\A_{(a-1)b_n(1+C_\cref{C:6}|S|^{-1/4})}(B[n])$ such that
$$\bbP\Big( \A_{(a-1)|S|}(S) \subseteq \A_{(a-1)b_n(1+C_\cref{C:6}|S|^{-1/8})}(B[n]) \Big)\ge 1-C\exp(-|S|^{\delta_5}).$$
\end{lemma}

\begin{proof}
 %
  We may assume $|S|$ is sufficiently large. Recall the definition of subset \rIDLA on page \pageref{page:censored} and the natural coupling of $A_i(S;b_n)$ and $A_i(B[n])$, with the property that $A_i(S;b_n)\subseteq A_i(B[n])$. Examine first the first $b_n|S|^{-1/2}$ particles added to $\A_i(B[n])$. Each of these is added to $\A_i(S;b_n)$ with probability at least $|S|/(b_n+b_n|S|^{-1/2})$. A Chernoff bound therefore shows that
\[
\bbP\bigg(|A_{b_n|S|^{-1/2}}(S;b_n)|-|S|>\frac{|S|^{1/2}}{1+|S|^{-1/2}}-|S|^{3/8}\bigg)>1-C\exp(-c|S|^{1/8}).
\]
By repeating this procedure at most $|S|^{1/2}$ times (here we use the assumption that $a\leq 2$), we get that on an event of probability at least 
$$1 - C|S|^{1/2}\exp(-C|S|^{1/8}) \ge 1-C\exp(-|S|^{1/9}),$$ 
we have $|\A_{ (a-1)b_n(1+C_\cref{C:6}|S|^{1/8})}(S;b_n)|>(a-1)|S|$. This finishes the lemma: we construct the coupling by letting $A_i(S)=A_{j(i)}(S;b_n)$ where $j(i)$ is the first time that $|A_{j(i)}(S;b_n)|=|S|+i$ and then with probability at least $1-C\exp(-|S|^{1/9})$ we have $j((a-1)|S|)\le (a-1)b_n(1+C_\cref{C:6}|S|^{1/8})$ so
\begin{align*}
  \A_{(a-1)|S|}(S)&=\A_{j((a-1)|S|)}(S;b_n)\subseteq
  \A_{(a-1)b_n(1+C_\cref{C:6}|S|^{1/8})}(S;b_n)\\
  &\subseteq \A_{(a-1)b_n(1+C_\cref{C:6}|S|^{1/8})}(B[n]).
\end{align*}
As needed.
\end{proof}

We will now prove that the \rIDLA started from a ball is contained in a suitable ball with high probability.

\begin{proposition} \label{converse}There exist $ C_\newC{C:7} >0$ such that for any $1 \leq a \leq 2$ the event
$$\A_{(a-1)b_n}(B[n]) \subseteq B\big[na^{1/d}(1+C_\cref{C:7}n^{-1/5})\big] $$ 
occurs with superpolynomially large probability.
\end{proposition}
Here and below, when we say about a sequence of events $E_n$ that they ``occur with superpolynomially large probability'' we mean that there exists a function $\phi$ decreasing to $0$ faster than any power of $n$ such that $\prob(E_n)>1-\phi(n)$. We might also use the phrase ``$\bbP(E_n)$ grows superpolynomially'' (and we do not insinuate by that the the probabilities increase as a function of $n$, just the bound above).

In order to prove this proposition, we first remark that as a consequence of Corollary~\ref{lose some}, the ball of radius $na^{1/d}(1-C_\cref{C:inner2}n^{-1/4})$ is included in the \rIDLA cluster $\A_{(a-1)b_n}(B[n])$ with stretched exponentially small probability. Hence we only need to control a number of particles of order $an^{d-1/4}$. However, these particles could in principle cover a thin spike that would reach very far. We know this cannot happen in regular IDLA, but in our case, a new particle may start on the furthermost point of the cluster, which complicates the situation. We therefore need to consider the genealogy of the particles in the process.

Recall that a {\em rooted tree} is a graph with no cycle and one marked point called the root. 
A {\em rooted forest} is a family of disjoint rooted trees. 

We construct the \rIDLA starting from a set $S$ in a new fashion. 
Consider a rooted forest whose vertices are indexed by integers and constructed as follows. At time $0$, $\T_0(S)$ is given by $|S|$ isolated sites indexed by $1,2,\dots,|S|$, which are the roots of the trees. 
At each step the vertex set of $\T_k$ is $\A_k(S)$ and the edges of $\T_k$ are constructed inductively as follows: $\T_{k+1}$ has all the edges of $\T_k$ and one more, from the starting point of the random walk which constructed $\A_{k+1}$ to its end i.e.\ to $\A_{k+1}\setminus\A_k$.
We will call this construction the \emph{genealogical construction} of the \rIDLA cluster, and $\T_k$ the \emph{genealogical tree} encoding it. We will look closely at the forest structure of $\T_k$, not at its embedding in $\mathbb{Z}^d$: rather, we think of a particle in the cluster as having both a position in $\mathbb{Z}^d$ and a position in the \emph{genealogical tree (or forest)} associated with the cluster. 

We start by an elementary lemma which is a generalisation of  \cite[Lemma 2.1]{FK}. As in First Passage Percolation, we attribute to every edge of the forest a geometric random variable with parameter $1/2$, independent of the random variables of other edges. We define the \emph{passage time} between two vertices as the sum of the random variables over edges on the geodesic between those two vertices (note that in this case, there is only one choice for the minimal path). The \emph{reaching time} of a vertex is the passage time between the root and this vertex.
\begin{lemma}\label{lem:distance}
Let $n,h>0$. Consider $\T_n(\{0\})$ constructed as above when starting from $A=\{0\}$. There exist $c,C>0$ such that for any $h\ge C\log n$, then
$$\bbP[\exists\text{ a vertex with reaching time larger than $h$}]\le e^{-ch}.$$
\end{lemma}

\begin{proof}

Let us first consider a slightly different model. Let $\tilde\T_t$ be the tree obtained from the same rule as for $\T_n$, but in continuous time (meaning that a new edge appears on each vertex according to an exponential clock of mean 1). Rather than explicitly writing the coupling between $\tilde\T_t$ and $\T_n$, we embed $\tilde\T_t$ in our probability space so that it is independent from $\T_n$. This model is exactly the model studied in \cite{FK}. In particular, if $X_t(k)$ is the number of vertices at graph distance $k$ from the root, Lemma 2.1 of \cite{FK} shows that
$$\bbE[X_t(k)]=\frac{t^k}{k!}.$$
Choose $C_\newC{C:11}>1$ large enough and $c_\newc{c:3.3.1}>0$ small enough so that $\bbE[X_t(k)]\le e^{-c_\cref{c:3.3.1}k}$ for $k\ge C_\cref{C:11}t$. Let $\tilde D_t(h)$ be the number of sites with passage time larger than $h$ in $\tilde \T_t$. We find
$$\bbE[\tilde D_t(h)]=\sum_{k=0}^\infty p_{k,h} \bbE[X_t(k)]=\sum_{k=0}^\infty p_{k,h}\frac{t^k}{k!},$$
where $p_{k,h}$ is the probability that the sum of $k$ independent geometric random variables of mean 1/2 is larger than $h-k$. There exists $c_\newc{c:3.3.2}>0$ such that $p_{k,h}\le e^{-c_\cref{c:3.3.2}h}$ for any $k\le h/3$. For simplicity, let us assume that $c_\cref{c:3.3.2}<2/3$.
By dividing the sum between $k\le c_\cref{c:3.3.2}h/2$ and $k\ge c_\cref{c:3.3.2}h/2$, we find that for $h\ge 2C_\cref{C:11}t/c_\cref{c:3.3.2}$,
\begin{align*}
\bbE[\tilde D_t(h)]&
=\sum_{k=0}^{c_\cref{c:3.3.2}h/2} p_{k,h}\frac{t^k}{k!}
+\sum_{k=c_\cref{c:3.3.2}h/2}^{\infty} p_{k,h}\frac{t^k}{k!}\\
&\le \sum_{k=0}^{c_\cref{c:3.3.2}h/2} p_{k,h}\frac{t^k}{k!}
+\sum_{k=c_\cref{c:3.3.2}h/2}^{\infty} \frac{t^k}{k!}\\
&\le \frac{c_\cref{c:3.3.2}h}{2}e^{c_\cref{c:3.3.2}h/2C_\cref{C:11}}e^{-c_\cref{c:3.3.2}h}
+\frac{e^{-c_\cref{c:3.3.1}c_\cref{c:3.3.2}h/2}}{1-e^{-c_\cref{c:3.3.1}}}.\end{align*}
In the second line, we used that $p_{k,h}\le 1$, and in the third line both the bound on $\frac{t^k}{k!}\le e^{-c_\cref{c:3.3.1}k}$ obtained by assuming that $k\ge c_\cref{c:3.3.2}h/2\ge C_\cref{C:11}t$, and the bounds $\frac{t^k}{k!}\le e^{t} \le e^{c_\cref{c:3.3.2}h/2C_\cref{C:11}}$ and $p_{k,h}\le e^{-c_\cref{c:3.3.2}h}$ when $k\le c_\cref{c:3.3.2}h/2$.

It only remains to go back from continuous time to discrete time. Let $D_n(h)$ be the number of sites with reaching time larger than $h$ in $\T_n$. From our construction, conditionally on the event $\{ \tilde \T_t \text{ has $k$ sites} \}$, $\tilde D_t $ has the same law as $D_k$.

\newcnotext{c:3.3.3}

Since with probability at least $1/2$, the aggregate $\tilde \T_{2\log n}$ has more than $n$ particles, we deduce that
\begin{align*}\bbP[ D_n(h)>0]&\le \bbE[D_n(h)]\le2\bbE[D_n(h)]\bbP(\tilde \T_{2\log n}\text{ has more than $n$ sites})\\
&\le 2\sum_{k\ge n}\bbE[D_k(h)]\bbP(\tilde \T_{2\log n}\text{ has $k$ sites})\\
&\le 2\sum_{k\ge n}\bbE[\tilde D_{2 \log n}(h)|\tilde \T_{2\log n}\text{ has $k$ sites}] \bbP(\tilde \T_{2\log n}\text{ has $k$ sites})\\
&\le 2\bbE[\tilde D_{2\log n}(h)]\le \exp(-c_\cref{c:3.3.3} h)
\end{align*}
for any $h\ge (4C_\cref{C:11}/c_\cref{c:3.3.2})\log n$, and $c_\cref{c:3.3.3}$ sufficiently small. 
\end{proof}

We are now ready to prove Proposition \ref{converse}. Recall that it stated that with superpolynomially large probability, $\A_{(a-1)b_n}(B[n])\subset B[a^{1/d}n(1+C_\cref{C:7}n^{-1/9d})]$.
\begin{proof}
By Corollary \ref{lose some}, we know that with stretched exponential probability, $\A_{(a-1)b_n}(B[n])$ contains $B[a^{1/d}n(1-C_\cref{C:inner2}n^{-1/4})]$. But that leaves only $C_\newC{C:unaccounted}n^{d-1/4}$ particles unaccounted for. 
Consequently, there are at most $C_\cref{C:unaccounted} n^{3d/4}$ particles outside $B[a^{1/d}n]$ at the end of the construction and therefore also at every previous step.

Recall that Lemma \ref{lem:exit prob} states that it is difficult to traverse any annulus $B[m]\setminus B[n]$ containing less than $\ep b_{m-n}$ vertices. In our setting, this means that there exists some constant $\beta_d$ such that for each of the annuli 
$$R_k=B[n+(k+1)\beta_dn^{3/4}]\setminus B[n+k\beta_dn^{3/4}],$$
the conclusion of Lemma \ref{lem:exit prob} holds in this annulus, with $p_{\textrm{Lemma \ref{lem:exit prob}}}=1/2$, if it is filled with less than $C_\cref{C:unaccounted} n^{d-{1/4}}$ particles (note that $3d/4 \leq d - 1/4$ since $d\geq 2$). Remark that $\beta_d$ is a constant that depends only on the dimension.

Since all the $R_k$'s are outside $B[n]$, each of them contains at most $C_\cref{C:unaccounted} n^{d-{1/4}}$ particles at any point in the construction of the cluster. Hence, the number of annuli that a particle can cross between its starting point and its exit point is stochastically dominated by a geometric variable of parameter $1/2$, and all these geometric variables can be taken to be independent.

The above discussion shows that a single particle may not go further than $C\log n$ annuli from its starting point. To get from this a bound on the size of the aggregate is a question about the forest $\T$. Precisely, the maximum $k$ that we consider is stochastically dominated by the maximum reaching time in the forest $\T_{(a-1)b_n}(B[n])$.

We now apply Lemma \ref{lem:distance}. Recall that it stated that for $\T_{b_n}(\{0\})$, the probability that it has a vertex with reaching time bigger than $\log(n)^2$ is smaller than $\exp\big(-c\log(n)^2\big)$. For every $x\in B[n]$, the corresponding tree in $\T_{(a-1)b_n}(B[n])$ is stochastically dominated by $\T_{b_n}(\{0\})$ (recall that $a\le 2$) so we get that, with superpolynomially large probability, the reaching time of every $x$ in every tree of $\T_{(a-1)b_n}(B[n])$ is smaller than $\log(n)^2$.

Now, the reaching time was defined using geometric random variables independent of the forest $\T_{(a-1)b_n}(B[n])$, so we can use the number of $R_k$ crossed by the corresponding particles, because the events that ``there are at most $C_\cref{C:unaccounted}n^{d-1/4}$ particles outside $B[a^{1/d}n]$ and yet our particle crossed annulus $i$'' have probability bounded above by $\tfrac 12$, independently of the tree structure. 

We conclude that with superpolynomially large probability, the annulus $R_{\log(n)^2}$ is not reached by any particle, so that $\A_{(a-1)b_n}(B[n]) \subseteq B[na^{1/d}(1+\beta_d\log(n)^2n^{-1/4}]$, which concludes the proof.
\end{proof}

Proposition \ref{converse} and Lemma \ref{lem:increasing} together imply the following corollary, which is a converse to Corollary \ref{lose some} :
\begin{corollary}\label{lose some ext}
There exist $C_{\newC{C:10}}>0$ such that for any set $S \subseteq B[n] $, 
$$\bbP \Big(\forall a\ge 1 \; 
  \A_{(a-1)|S|}(S) \subseteq B([na^{1/d}(1+C_\cref{C:10}n^{-1/5}])\Big)$$
grows superpolynomially in $|S|$.
\end{corollary}
This follows from Lemma \ref{lem:increasing}, which states that 
$$\A_{(a-1)|S|}(S)\subset\A_{(a-1)b_n(1+Cn^{-1/2})}(B[n])$$ 
and Proposition \ref{converse} for $a\le 2$. Iterating (as in the end of the proof of Corollary \ref{lose some}) gives the result for general $a$. We omit the details.




\section{Proof of Theorem~\ref{thm:main}}
\subsection{Inner bound}
Our proof requires rough initial bounds before better bounds can be proved. Our rough outer bound is the obvious remark that $\A_n \subset B[n]$ because it is connected. 
For a rough inner bound, we have the following lemma:

\begin{lemma}\label{initiate}
There exists $\delta_6>0$ such that or all $n$ big enough,
$$\bbP\Big(B[n/2] \subseteq \A_{e^{b_n}}\Big)\geq 1-\exp(-n^{\delta_6}).$$
\end{lemma}
\begin{proof}
We consider the particles that start from the origin. At step $k$, the event that the new particle starts from the origin has probability $1/k $, hence the number of particles started from the origin by time $e^{b_n}$ has expected value $b_n$ and is bigger than $2b_n/3$ with probability $e^{-b_n/18}$ from a Chernoff bound. Classical \IDLA results (in particular, the explicit bound in \cite[paragraph 3.1.3]{asselah2010sub} is much stronger than what is needed here) guarantee that the standard \IDLA with $2b_n/3$ particles started from the origin covers at least $B([n/2])$ except on an event of stretched exponentially small probability.
\end{proof}

Now that we have a rough bound, we are in a position to improve it. The following proposition states that a rough inner bound can always be improved, provided we have an outer bound as well.

\begin{proposition}\label{win always}
Let $\lambda>1$ and $\ep >0$ be two parameters. Suppose that $B[n] \subseteq S \subseteq B[\lambda n]$ and that $|S|>(1+\ep)b_n$. For a constant $\eta_2$ depending only on the dimension, and uniformly in $S$,
$$B\left[3\lambda n\Big(1+\frac{\eta_2\ep}{\lambda^d}\Big)^{1/d}\right]\subseteq \A_{((3\lambda)^d-1)|S|}(S)$$
with superpolynomially large (in $\ep n$) probability.
\end{proposition}

There are two steps in the proof. We first look at the growth of $B[n]$ while ignoring completely the sites in $S\setminus B[n]$. Then, we use sites of $S\setminus B[n]$ (which are not too far from the origin). These sites represent a tiny proportion of $\A_{((3\lambda)^d-1)|S|}(S)$, but it is more than sufficient to counter the loss of the first step. Lemma~\ref{crucial lemma} is crucial in this argument.
\begin{proof}
We know from Corollary~\ref{lose some} that the ball of radius $3\lambda n(1-C_\cref{C:inner2}n^{-1/4})$ is included in the subset \rIDLA aggregate $\A_{((3\lambda)^d-1)|S|}(B[n];|S|)$ with probability greater than $1-\exp(-n^{\delta_3})$.

Lemma~\ref{crucial lemma} now yields that the aggregate we are interested in stochastically dominates the one built by adding the particles of $S\setminus B[n]$ to the subset \rIDLA aggregate $\A_{((3\lambda)^d-1)|S|}(B[n];|S|)$. In a formula,
\begin{align*}
&\bbP\Big(E
  \subseteq \A_{((3\lambda)^d-1)|S|}(S) \Big) \\
& \geq \bbP\Big(E
  \subseteq \DA_{S\setminus B[n]}\Big(\A_{((3\lambda)^d-1)|S|}(B[b];|S|)\Big) \Big) \\
& \geq \bbP\Big(E
  \subseteq \DA_{S\setminus B[n]}\Big(B[3\lambda n(1-C_\cref{C:inner2}n^{-1/4})]\Big) \Big) 
   - \exp(-n^{\delta_3}).
\end{align*}
(this holds for any set $E$ but, for the curious, we will eventually use it for $E=B\big[3\lambda n(1+\eta_2\ep/\lambda^d)^{1/d}\big]$, i.e.\ the set from the statement of the lemma).

Next, since $S \subset B[\lambda n]$, all the points in $S\setminus B[n]$ are inside the half-radius of $B[3\lambda n(1-C_\cref{C:inner2}n^{-1/4})]$, and we are in a position to apply Lemma~\ref{win some}. It yields the following:
\begin{align*}
& \bbP\Big(E
  \subseteq \DA_{S\setminus B[n]}\Big(B[3\lambda n(1-C_\cref{C:inner2}n^{-1/4})]\Big) \Big) \\
& \geq \bbP\Big(E\subseteq \DA_{\kappa}\Big(B[3\lambda n(1-C_\cref{C:inner2}n^{-1/4})]\Big)\Big),
\end{align*}
where $\kappa$ is a random variable following a binomial law with  $|S|-b_n >\ep b_n$ trials and probability of success $\eta$, and $\eta>0$ is the constant defined in Lemma~\ref{win some}.

Now, applying Chernoff's bound yields that $\kappa > \frac{3}{4}\eta \ep b_n$ with probability higher than $1-\exp(-c\ep b_n)$. This means that the number of particles added is not too small. This fact, together with the inner bound for standard \IDLA\ (from \cite{asselah2010sub,jerison2010logarithmic} once again), guarantees that with exponentially high probability, $\DA_{\kappa}\Big(B[3\lambda n(1-C_\cref{C:inner2}n^{-1/4})]\Big)$ contains a ball of radius $3\lambda n(1+\eta\ep/(8\cdot 3^{d-1}\lambda^d))^{1/d}$. The lemma thus holds with the value $\eta_2=\eta/(8\cdot 3^{d-1})$.
%
%
%
\end{proof}

This method for improving inner bounds enables us to prove the inner part of Theorem~\ref{thm:main}.

\begin{proposition}\label{sharp inner}
Let $d\geq 2$. There exists constants $c_\cref{c:thm}$, $C_\cref{C:Thm}$ depending only on the dimension such that almost surely,
$$B\Big[n(1-C_\cref{C:Thm}n^{-c_\cref{c:thm}} ) \Big] \subseteq \A_{b_n},$$
for $n$ large enough.
\end{proposition}

\begin{proof}
Lemma~\ref{initiate} and the remark before it provide us with the following bounds: for arbitrarily large $m_0$, with probability at least $1-\exp(-m_0^{\delta_6})$, 
$B[m_0]\subseteq \A_{e^{2b_{m_0}}} \subseteq B[e^{2b_{m_0}}].$

Corollary \ref{lose some ext} then guaranties that conditionally on the previous event, with superpolynomially (in $m_0$) large probability,
$$
\A_{(a^{d}-1)e^{2b_{m_0}}} 
\subseteq B[ae^{2b_{m_0}}(1+C_\cref{C:10}e^{-2b_{m_0}/9d})]
$$ 
for any $a$. In other words, for some $\tau=\tau(m_0)$, $\A_{n}\subseteq B[\tau n^{1/d}]$ for all $n$. Let us also assume that $\tau$ is sufficiently large so that $b_{\tau n^{1/d}}\ge 2n$ (though it would have probably held even if we had not assumed it explicitly).

We now repeatedly apply Proposition~\ref{win always}, starting from $m_0$. Recall that it states that if $B[r]\subseteq \A_M\subseteq B[\lambda r]$ and if $M>(1+\ep)b_r$, then with high probability, for some other $r'$ and $M'$ we have $B[r']\subset \A_{M'}$. So applying the proposition repeatedly gives sequences $r_i$ and $M_i$ such that, with high probability,  $B[r_i]\subseteq \A_{M_i}$. Let us list all relevant parameters:
\begin{align*}
r_0&=m_0&r_{i+1}&=3\lambda_ir_i\Big(1+\frac{\eta_2\ep_i}{\lambda_i^d}\Big)^{1/d}\\
M_0&=e^{2b_{m_0}}&M_{i+1}&=(3\lambda_i)^dM_i\\
\ep_i&=M_i/b_{r_i}-1&\lambda_i&=\tau M_i^{1/d}/r_i
\end{align*}
The constant $\eta_2$ comes from Proposition \ref{win always}, but we may assume that it is small enough, so let us assume $\eta_2\le \tfrac 12$. 

Let us now analyse these parameters. We first note that $\lambda_i$ is decreasing --- indeed, $M_i^{1/d}$ is increased at each step by $3\lambda_i$ while $r_i$ is increased by more. On the other hand, we always have $B[r_i]\subseteq A_{M_i}$ and $b_{\tau M_i^{1/d}}\ge 2M_i$ so 
\[
\lambda_i^d=\Big(\frac{\tau M_i^{1/d}}{r_i}\Big)^d=\frac{b_{\tau M_i^{1/d}}}{b_{r_i}}
\ge \frac{2M_i}{M_i}=2
\]
so $\lambda_i\ge 2^{1/d}$ for all $i$. 

More important is the behaviour of $\ep_i$. Putting together the formulas for $r_{i+1}$ and $M_{i+1}$ gives
\[
\ep_{i+1}+1=\frac{M_{i+1}}{b_{r_{i+1}}} = 
\frac{(3\lambda_i)^dM_i}{(3\lambda_i)^db_{r_i}(1+\eta_2\ep_i/\lambda_i^d)}=
\frac{\ep_i+1}{1+\eta_2\ep_i/\lambda_i^d}
\]
rearranging gives
\[
\ep_{i+1}\le \ep_i-\frac{c\ep_i}{\lambda_i^d}
\]
and since $\lambda_i$ is bounded above, we get that $\ep_i$ decreases exponentially in $i$. 

On the other hand, $\ep_i$ does not decrease too fast: since we assumed $\eta_2\le \tfrac 12$, and since $\lambda_i^d\ge 2$ we get that $\ep_{i+1}\ge \min\{1,\tfrac 12 \ep_i\}$. Since $M_{i+1}\ge 3M_i$, we get that $M_i$ increases faster than $\ep_i$ decreases. This is important because the bad event of Proposition \ref{win always} happens with probability superpolynomially small in $M_i\ep_i$. Thus we have just shown that these bad events have summable probabilities, and further, the sum is superpolynomially small in $m_0$. 

This establishes the proposition on a sequence. Indeed, since $M_i$ increases no more than exponentially, then we get $\ep_i\le CM_i^{-c}$. So we get $B[r_i]\subseteq A_{b_{r_i}(1-r_i^{-c})}$, which is equivalent to the claim. To extend from a subsequence to all integers, we use Corollary \ref{lose some}. We get that between $M_i$ and $M_{i+1}$ the contained ball still follows the volume of the aggregate up to a polynomial probability, with superpolynomially large probability (in $M_i$). These probabilities may be summed.

All in all we get that for some $c_\newc{c:twice}$ independent of $m_0$ and some $C(m_0)$, the event
\[
B[n(1-C(m_0)n^{-c_\cref{c:twice}})]\subset A_{b_n}\qquad\forall  n
\]
holds with superpolynomially large probability. This means that almost surely, it does indeed hold for some $m_0$, and this means that the proposition holds with $c_\cref{c:thm}=\tfrac 12 c_\cref{c:twice}$, and an arbitrary $C_\cref{C:Thm}$. 
\end{proof}

\subsection{Outer bound}

\begin{proposition}\label{thm:sub2}
  Almost surely, $A_{b_n}\subseteq  B\big[n(1+Cn^{-c})\big]$.
\end{proposition}
\begin{proof}The proof is identical to the proof of Proposition \ref{converse}, but using the inner bound (Proposition \ref{sharp inner}) as a basis. Let us recall quickly the argument. Proposition \ref{sharp inner} ensures that $B[n(1-Cn^{1-c})]\subset\A_{b_n}$ but that leaves only $Cn^{d-c}$ particles unaccounted for (and possibly outside $B[n]$). Hence the same holds during the entire process up to time $b_n$.
Lemma \ref{lem:exit prob} then ensures that annuli of width $Cn^{1-c}$ around $B[n]$ are difficult to cross for a random walker, hence none of the particles cross more than $\log n$ of them. Finally, Lemma \ref{lem:distance} ensures that none of the trees of $\T_{b_n}(\{0\})$ has depth larger than $\log^2n$, so no particle may end up further than $Cn^{1-c}\log^3 n$. This ends the proof.
\end{proof}

\newpage
\appendix

\section*{Appendix : Proof of Lemma \ref{st-exp-bound}}
The following is a guide on how to read \cite{asselah2010logarithmic} and modify the authors' proof to obtain the desired result. Note than any reference used in the following is to be found inside \cite{asselah2010logarithmic}. First, we start with an already occupied region $B[n]$. The flashing process introduced by the authors is the same, except it only starts flashing after exiting $B[n]$, and the coupling they describe holds true in our setting. Therefore an interior bound like the one we wish to prove can be proven in the flashing process setting. Then, in paragraph 4.2, we will use $\xi = nr^{3/2}$ (instead of $\log(n)$) in equation (4.10). Moreover we follow the authors' recommandation to use a constant $h_k=h$ when dealing with \IDLA.

Recall that $W_k(\mathcal{T})$ is the number of unsettled explorers (out of our $N$ initial explorers) that stand in a cell $\mathcal{T}$ when the cluster is built up to radius $r_k$, and that $M_k(\mathcal{T})$ is the number of explorers that exit $B[r_k]$ through $\mathcal{T}$. Bounding the number of settled explorers that exit $B[r_k]$ through $\mathcal{T}$ by saying at most one can have settled on each site of $B[r_k-h]\setminus B[n]$, we get the usual equation:
$$ W_k(\mathcal{T})+ L_k(\mathcal{T}) \geq M_k(\mathcal{T}), $$
with $L_k(\mathcal{T})$ the number of particles that exit $B[r_k]$ through $\mathcal{T}$ when one is started on each site of $B[r_k-h]\setminus B[n]$. Note that here we stray from the authors' proof since our $L_k$ is not the same as theirs. Hence we need to check that we still have the desired value for $\mathbb{E}\left[ M_k(\mathcal{T}) - L_k(\mathcal{T}) \right]$. The computation is done in a subparagraph called {\it Step 1 } which we emulate here. Recall that for an integer-valued function $f$, $M( f , r_k, \mathcal{T})$ is the number of particles that exit $B[r_k]$ through $\mathcal{T}$ when $f(x)$ particles are started at each point $x$ independently. We follow the authors in also denoting $M( A , r_k, \mathcal{T}) = M( \mathbf{1}_A , r_k, \mathcal{T})$ when $A$ is a subset of $\mathbb{Z}^d$.
\begin{eqnarray*}
\mathbb{E}\left[ M_k(\mathcal{T}) - L_k(\mathcal{T}) \right] &=& \mathbb{E}\left[ M(N \textbf{1}_0, r_k, \mathcal{T}) \right] - \mathbb{E}\left[ M( B[r_k-h]\setminus B[n], r_k, \mathcal{T}) \right] \\
& =& \mathbb{E}\left[ M((b_n+N-b_{r_k-h}) \textbf{1}_0, r_k, \mathcal{T}) \right] \\
&& + \mathbb{E}\left[ M(B[n], r_k, \mathcal{T}) \right]- \mathbb{E}\left[ M(b_n \textbf{1}_0, r_k, \mathcal{T}) \right] \\
&& + \mathbb{E}\left[ M(b_{r_k-h} \textbf{1}_0, r_k, \mathcal{T}) \right] - \mathbb{E}\left[ M(B[r_k-h], r_k, \mathcal{T}) \right]
\end{eqnarray*}

The absolute values of the differences on the second and third lines are then each bounded using Corollary 5.4, so that we get for some constants $C, \kappa$ depending only on the dimension,
\begin{eqnarray*}
\mathbb{E}\left[ M_k(\mathcal{T}) - L_k(\mathcal{T}) \right] & \geq & \left( b_n + N - b_{r_k-h} \right) \mathbb{P}_0(S(H_k)\in \mathcal{T}) - C \\
& \geq & \left( b_n + N - b_{r_k-h} \right) \frac{\kappa}{r_k^{d-1}}.
\end{eqnarray*}
Note that as long as $r_k-h \leq n(1+r-r^{3/2})^{1/d}$, the first factor on the right hand side (that is, the difference in volumes) is bigger than $\kappa_d n^d r^{3/2}$ for a constant $\kappa_d$ depending only on the dimension.

We then resume the course of the authors' proof, noting that the bound (4.14) does not concern us since the order of the left-hand-side is at most logarithmic (see (4.18)). Finally, we conclude in (4.15) by replacing the authors' $\log(n)$ with $nr^{3/2}$.

\newpage
\paragraph{Acknowledgements}The authors would like to thank Vincent Beffara and Vladas Sidoravicius for many fruitful discussions, and Nicolas Curien for referring us to \cite{FK}. This paper was partly written during the visit of the second and fourth authors to the Weizmann Institute in Israel. The first author is
the incumbent of the Renee and Jay Weiss Professorial Chair. The
second author was supported by the EU Marie-Curie RTN CODY, the ERC AG
CONFRA, as well as by the Swiss {FNS}. The
third author was supported by the Israel Science Foundation.

\begin{flushright}
\footnotesize\obeylines
  \textsc{Weizmann Institute}
  \textsc{Rehovot, Israel}
  \textsc{E-mail:} \texttt{itai.benjamini@weizmann.ac.il ; gady.kozma@weizmann.ac.il }
$ $\\
  \textsc{Universit\'e de Gen\`eve}
  \textsc{Gen\`eve, Switzerland}
  \textsc{E-mail:} \texttt{hugo.duminil@unige.ch}
$ $\\
  \textsc{LPSM, Universit\'e Denis Diderot}
  \textsc{Paris, France}
  \textsc{E-mail:} \texttt{lucas@lpsm.paris}
\end{flushright}

\end{document}